\documentclass[a4paper,12pt]{article}

\usepackage[intlimits]{amsmath}
\usepackage{amssymb}
\usepackage{amsthm}
\usepackage{mathrsfs}
\usepackage{cite}
\usepackage{mathrsfs}
\usepackage{color}

\newcommand{\pr}{\partial}
\newcommand{\dom}{\Omega}

\newcommand{\re}{\mathfrak{Re}\;}

\newcommand{\R}{\mathbb{R}}
\newcommand{\C}{\mathbb{C}}
\newcommand{\F}{\mathscr{F}}
\newcommand{\G}{\mathcal{G}}
\newcommand{\T}{\mathcal{T}}
\newcommand{\X}{\mathscr{X}}
\newcommand{\Y}{\mathscr{Y}}
\newcommand{\el}{\mathcal{L}}
\newcommand{\Ord}{\mathscr{O}}
\newcommand{\ttr}{\mathfrak{t}}
\newcommand{\m}{\mathfrak{m}}

\newcommand{\dd}{\,\text{d}}
\newcommand{\supp}{\mbox{supp\;}}
\newcommand{\crl}{{\nabla\times}}
\newcommand{\dv}{{\nabla\cdot}}

\newtheorem{thm}{Theorem}[section]
\newtheorem{prop}{Proposition}[section]
\newtheorem{lem}{Lemma}[section]
\newtheorem{cor}{Corollary}[section]

\title{On an inverse boundary value problem for a nonlinear time harmonic Maxwell system}

\author{C\u{a}t\u{a}lin I. C\^{a}rstea\thanks{HKUST Jockey Club Institute for Advanced Study, The Hong Kong University of Science and Technology, Clear Water Bay, Kowloon, Hong Kong; email: catalin.carstea@gmail.com}}

\date{}

\begin{document}
\maketitle

\begin{abstract}
This paper considers a class of nonlinear time harmonic Maxwell systems at fixed frequency, with nonlinear terms taking  the form $\mathscr{X}(x,|\vec E(x)|^2)\vec E(x)$, $\mathscr{Y}(x,|\vec H(x)|^2)\vec H(x)$, such that $\mathscr{X}(x,s)$, $\mathscr{Y}(x,s)$ are both real analytic in $s$. Such nonlinear terms appear in nonlinear optics theoretical models. Under certain regularity conditions, it can be shown that boundary measurements of tangent components of the electric and magnetic fields determine the electric permittivity and magnetic permeability functions as well as the form of the nonlinear terms.
\end{abstract}

\section{Introduction}

Let $\dom\subset\R^3$ be a bounded domain with smooth boundary. The (macroscopic) Maxwell's equations for the electromagnetic field in a material filling the domain $\dom$, without (macroscopic) densities of charge or current, are
\begin{equation}\label{maxwell}
\left\{\begin{array}{l}
\crl\vec{\mathcal{E}}=-\pr_t\vec{\mathcal{B}},\quad
\crl\vec{\mathcal{H}} = \pr_t\vec{\mathcal{D}},\\[5pt]
\dv\vec{\mathcal{D}}=0,\quad
\dv\vec{\mathcal{B}}=0,\\[5pt]
\vec{\mathcal{D}}=\underline{\epsilon}\vec{\mathcal{E}}+\vec{\mathscr{P}}_{NL}(\vec{\mathcal{E}}),\\[5pt]
\vec{\mathcal{B}}=\underline{\mu}\vec{\mathcal{H}}+\vec{\mathscr{M}}_{NL}(\vec{\mathcal{H}}).
\end{array}\right.
\end{equation}
For a linear medium, $\vec{\mathscr{P}}_{NL}(\vec{\mathcal{E}})=\vec{\mathscr{M}}_{NL}(\vec{\mathcal{H}})=0$, and the system takes the familiar form that has been studied extensively both from the point of view of the forward problem and also of the inverse problem.  
Nonlinear effects have been observed in practice, as (for example) the extensive literature on nonlinear optics indicates. As an example, see \cite{M}, \cite{MHRW}, \cite{S}, where nonlinearities of the kind appearing in this paper are put forward. 

We will consider time-harmonic fields of the form\footnote{The $\ast$ denotes complex conjugation.}
\begin{equation}
\vec{\mathcal{E}}(t,x)=\vec E(x)e^{-i\omega t}+\vec E^\ast(x) e^{i\omega t},\quad
\vec{\mathcal{H}}(t,x)=\vec H(x)e^{-i\omega t}+\vec H^\ast(x) e^{i\omega t},
\end{equation}
where $\omega>0$ will be a given fixed frequency. At high frequency, the system \eqref{maxwell} may be taken to reduce to
\begin{equation}\label{eq}
\left\{\begin{array}{l}
\crl\vec E(x)= i\omega\mu(\omega, x) \vec H(x)+\Y(\omega,x,|\vec H(x)|^2)\vec H(x),\\[5pt]
\crl\vec H(x)=- i\omega\epsilon(\omega, x) \vec E(x)-\X(\omega,x,|\vec E(x)|^2)\vec E(x).
\end{array}\right.
\end{equation}
A common model is that of a Kerr-type nonlinearity:
\begin{equation}
\X(x,|\vec E(x)|^2)\vec E(x)=a(x)|\vec E(x)|^2\vec E(x).
\end{equation}
The inverse problem for a model in which both $\X$ and $\Y$ have this form has been investigated in \cite{AZ}.
However, more realistic models feature a saturation effect for $\X$ when the field intensity is high (see \cite{M}, \cite{MHRW}). One example, given in \cite{S}, is
\begin{equation}
\X(x,|\vec E(x)|^2)\vec E(x)=\frac{a(x)|\vec E(x)|^2}{1+b(x)|\vec E(x)|^2}\vec E(x).
\end{equation}
A more complicated model is deduced in \cite{MHRW}.

In this paper it will be assumed that $\X(x,s)$, $\Y(x,s)$ are analytic in $s$, having expansions at zero 
 \begin{equation}
 \X(x,s)=\sum_{k=1}^\infty a_k(x) s^k,\quad  \Y(x,s)=\sum_{k=1}^\infty b_k(x) s^k,
 \end{equation}
 and
 \begin{equation}\label{assumption-1}
 \epsilon,\mu\in C^5(\dom),\quad a_k,b_k\in C^1(\dom) 
 \end{equation}
 \begin{equation}\label{assumption-2}
 \re \epsilon,\re\mu>\lambda>0,
 \end{equation}
 \begin{equation}\label{assumption-3}
  ||\epsilon||_{W^{5,\infty}(\dom)},||\mu||_{W^{5,\infty}(\dom)} <M<\infty,
 \end{equation}
 \begin{equation}\label{assumption-4}
 \sum_{k=1}^\infty \left(||a_k||_{W^{1,\infty}(\dom)}+ ||b_k||_{W^{1,\infty}(\dom)}\right)s^k < M s,\quad \forall\;0<s<s_0,
 \end{equation}
 \begin{equation}\label{assumption-5}
 \sum_{k=1}^\infty k\left(||a_k||_{W^{1,\infty}(\dom)}+||b_k||_{W^{1,\infty}(\dom)}\right)s^{k-1} < M,\quad \forall\;0<s<s_0,
 \end{equation}
  \begin{equation}\label{assumption-6}
 \sum_{k=2}^\infty k(k-1)\left(||a_k||_{L^\infty(\dom)}+||b_k||_{L^\infty(\dom)}\right)s^{k-2} < M,\quad \forall\;0<s<s_0,
 \end{equation}
 where $\lambda$, $M$, $s_0$ are positive constants.
 
 A note on notation: in order to make equations easier to read, the explicit dependence on $x$ of various quantities will be suppressed. For example,  $\X(|\vec E|^2)$ will stand for $\X(x,|\vec E(x)|^2)$ or $\X(\cdot, |\vec E(\cdot)|^2)$.
 
 \subsection{The forward problem}
 
 We will say that a vector field belongs to $L^p(\dom)$, $W^{s,p}(\dom)$, etc. if each component belongs to those respective spaces.
Let
\begin{equation}
W^{s,p}_{div}(\dom)=\left\{\vec A\in W^{s,p}(\dom): \dv\vec A\in W^{s,p}(\dom)\right\},
\end{equation} 
with the natural choice of norms.
If  $\vec n$ is the outer unit normal to $\pr\dom$, let
\begin{equation}
TW^{s,p}(\pr\dom)=\left\{\vec A\in W^{s,p}(\pr\dom): \vec n\cdot \vec A=0\right\},
\end{equation}
\begin{equation}
TW^{s,p}_{div}(\pr\dom)=\left\{\vec A\in TW^{s,p}(\pr\dom): div(\vec A)\in W^{s,p}(\pr\dom)\right\},
\end{equation}
where $div(\vec A)$ is the divergence  associated with the metric  induced on the boundary by the Euclidean metric of $\R^3$. For a smooth vector field $\vec A$ on $\dom$, let
\begin{equation}
\ttr(\vec A)=-\vec n\times(\vec n\times\vec A|_{\pr\dom}),
\end{equation}
i.e. the component tangential to the boundary of the restriction of $\vec A$. $\ttr$ clearly extends to a bounded operator from $W^{s,p}(\dom)$ to $TW^{s-1/p,p}(\pr\dom)$. Let $W^{1,p}_{b}(\dom)=\ttr^{-1}\left(TW^{1-1/p,p}_{div}(\pr\dom)\right)$, with the norm
\begin{equation}
||\vec A||_{W^{1,p}_b(\dom)}=||\vec A||_{W^{1,p}(\dom)}+||\ttr(\vec A)||_{TW^{1-1/p,p}_{div}(\pr\dom)}.
\end{equation}
Finally, let
\begin{equation}
W^{1,p}_D(\dom)=\ttr^{-1}(0), \quad ||\cdot||_{W^{1,p}_D(\dom)}=||\cdot||_{W^{1,p}(\dom)}.
\end{equation}
 
 Before discussing the inverse problem a well-posedness result for the forward problem is necessary. In section \ref{forward} it will be proven that:
 
\begin{thm}\label{n-existence}
For $3< p\leq 6$ there exists a discrete set $\Sigma\subset\C$ and a constant $\m>0$ such that if $\omega\not\in\Sigma$ and $\vec f\in TW^{1-1/p,p}_{div}(\pr\dom)$, $||\vec f||_{TW^{1-1/p,p}_{div}(\pr\dom)}<\m$ there exists a unique solution $\bold U=(\vec E,\vec H)\in W^{1,p}_{b}(\dom)\times W^{1,p}_{b}(\dom)$ of the system
\begin{equation}
\left\{\begin{array}{l}
\crl\vec E= i\omega\mu \vec H + \Y(|\vec H|^2)\vec H,\\[5pt]
\crl\vec H=- i\omega\epsilon \vec E- \X(|\vec E|^2)\vec E,
\end{array}\right.
\end{equation}
such that $\ttr(\vec E)=\vec f$ and
\begin{equation}
||\vec E||_{W^{1,p}_{b}(\dom)}+||\vec H||_{W^{1,p}_{b}(\dom)}\leq C||\vec f||_{TW^{1-1/p,p}_{div}(\pr\dom)},
\end{equation}
where $C>0$ is a constant that does not depend on $\vec f$.
\end{thm}

The proof of this result follows from estimates for the linear system obtained in \cite{AZ} and a standard contraction principle argument.

\subsection{The inverse problem}

An inverse boundary value problem consists of the question of determining the interior physical properties of a possibly non-homogeneous object from measurements taken on the boundary of the object. A fundamental sub-problem is the question of uniqueness: if two objects of the same shape give the same boundary measurement data, does it follow that their (relevant) physical properties are identical in the interior?

For time-harmonic electromagnetic fields in media in which \eqref{eq} applies, in light of Theorem \ref{n-existence} we can then define the set of boundary measurements
\begin{multline}
\mathcal{B}_{\epsilon,\mu,\F}=\left\{(\ttr(\vec E),\ttr(\vec H))\in TW^{1-1/p,p}_{div}(\pr\dom)\times TW^{1-1/p,p}_{div}(\pr\dom)\right. \\ 
\left.: (\vec E, \vec H)\text{ is a solution of \eqref{eq}}           \right\}.
\end{multline}
In section \ref{inverse} we prove that
\begin{thm}\label{uniqueness}
Suppose $(\epsilon,\mu,\F)$ and $(\epsilon,\mu,\F)$ are as above, $\omega\not\in\Sigma\cup\Sigma'$, and $\mathcal{B}_{\epsilon,\mu,\F}=\mathcal{B}_{\epsilon',\mu',\F'}$. Then $(\epsilon,\mu,\F)=(\epsilon',\mu',\F')$.
\end{thm}

The inverse boundary value problem has been studied extensively in the linear case. See  for example \cite{SunU-2}, \cite{CP}, \cite{OPS}, \cite{OS}, \cite{COS}, \cite{C}, \cite{KSU}, \cite{Z} etc. Uniqueness results similar to Theorem \ref{uniqueness} for nonlinear equations have been obtained in \cite{IS}, \cite{IN}, \cite{Sun1}, \cite{SunU},\cite{I1}, \cite{I2}, \cite{HS}, \cite{Sun2} using a linearization method. Here we will follow an idea from \cite{AZ} and use the asymptotics in a small parameter $t$ of solutions of \eqref{eq} with boundary data $\ttr(\vec E)=t\vec f$ in order to inductively prove uniqueness for the coefficients of the nonlinearity. We will also need to use certain special solutions, so called geometric optics (CGO) solutions, which we construct following the method in \cite{C}.

 \section{The forward problem}\label{forward}

\subsection{Preliminaries}

The existence and uniqueness of $W^{1,p}$ solutions to the linear Maxwell system, for $p>2$, has been investigated in \cite{AZ}. First we quote an existence result for the boundary value problem for the homogeneous system:

\begin{thm}[see {\cite[Theorem 3.1]{AZ}}]\label{h-existence}
For $2\leq p\leq 6$ there exists a discrete set $\Sigma\subset\C$ such that if $\omega\not\in\Sigma$ and $\vec f\in TW^{1-1/p,p}_{div}(\pr\dom)$ there exists a unique solution $(\vec E,\vec H)\in W^{1,p}_{b}(\dom)\times W^{1,p}_{b}(\dom)$ of the system
\begin{equation}
\left\{\begin{array}{l}
\crl\vec E= i\omega\mu \vec H,\\[5pt]
\crl\vec H=- i\omega\epsilon \vec E,
\end{array}\right.
\end{equation}
such that $\ttr(\vec E)=\vec f$ and
\begin{equation}
||\vec E||_{W^{1,p}_{b}(\dom)}+||\vec H||_{W^{1,p}_{b}(\dom)}\leq C||\vec f||_{TW^{1-1/p,p}_{div}(\pr\dom)},
\end{equation}
where $C>0$ is a constant that does not depend on $\vec f$.
\end{thm}

We also need the following result for the inhomogeneous system:

\begin{thm}[see {\cite[Theorem 3.2]{AZ}}]\label{inh-existence}
For $2\leq p\leq 6$ exists a discrete set $\Sigma\subset\C$ such that if $\omega\not\in\Sigma$ and $\vec J_e, \vec J_m\in W^{0,p}_{div}(\dom)$, $\vec n\cdot\vec J_e|_{\pr\dom}, \vec n\cdot\vec J_m|_{\pr\dom} \in W^{1-1/p,p}(\pr\dom)$,  there exists a unique solution $(\vec E,\vec H)\in W^{1,p}_{D}(\dom)\times W^{1,p}_{b}(\dom)$ of the system
\begin{equation}
\left\{\begin{array}{l}
\crl\vec E= i\omega\mu \vec H+\vec J_m,\\[5pt]
\crl\vec H=- i\omega\epsilon \vec E-\vec J_e,
\end{array}\right.
\end{equation}
such that
\begin{multline}
||\vec E||_{W^{1,p}_{b}(\dom)}+||\vec H||_{W^{1,p}_{b}(\dom)}\leq C\left(||\vec J_e||_{W^{0,p}_{div}(\dom)}+||\vec J_m||_{W^{0,p}_{div}(\dom)}\right.\\
\left.  +||\vec n\cdot\vec J_e|_{\pr\dom}||_{W^{1-1/p,p}(\pr\dom)} + ||\vec n\cdot\vec J_m|_{\pr\dom}||_{W^{1-1/p,p}(\pr\dom)}    \right),
\end{multline}
where $C>0$ is a constant that does not depend on $\vec J_e, \vec J_m$.
\end{thm}

Under the conditions of Theorem \ref{inh-existence}, we will write
\begin{equation}
\begin{pmatrix}\vec E\\ \vec H  \end{pmatrix}=\G_{\epsilon,\mu}(\begin{pmatrix} \vec J_m\\ \vec J_e\end{pmatrix}).
\end{equation}
It is a corollary of Theorem \ref{inh-existence} that $\G_{\epsilon,\mu}$ is bounded from $W^{1,p}(\dom)\times W^{1,p}(\dom)$ to $ W^{1,p}_{D}(\dom)\times W^{1,p}_{b}(\dom)$.

For the sake of simplifying notation let:
\begin{equation}
\bold U=\begin{pmatrix}\vec E\\ \vec H  \end{pmatrix},\quad 
\el_{\epsilon,\mu}=\begin{pmatrix}\crl & -i\omega\mu\\ i\omega\epsilon &\crl \end{pmatrix},\quad
\F(\bold U)=\begin{pmatrix} \Y(|\vec H|^2)\vec H \\ -\X(|\vec E|^2)\vec E  \end{pmatrix}.
\end{equation}
Then equation \eqref{eq} can be written
\begin{equation}
\el_{\epsilon,\mu}\bold U=\F(\bold U).
\end{equation}
Given $\vec f\in TW^{1-1/p,p}_{div}(\pr\dom)$, let
\begin{equation}
\bold U_0=\begin{pmatrix}\vec E_0 & \vec H_0  \end{pmatrix}^t \in W^{1,p}_{b}(\dom)\times W^{1,p}_{b}(\dom)
\end{equation}
be the solution given in Theorem \ref{h-existence}. Then a solution of \eqref{eq} with the boundary condition $\ttr(\vec E)=\vec f$ would be a fixed point of the operator
\begin{equation}\label{T-def}
 \T_{\vec f, \epsilon,\mu}(\bold U)=\bold U_0+\G_{\epsilon,\mu}(\F(\bold U)).
\end{equation}

\subsection{Existence of solutions}

 From now we will only consider $p>3$. Then $W^{1,p}(\dom)\subset L^\infty(\dom)$ and there exists a constant $c>0$ such that
 \begin{equation}
 ||\bold U||_{L^\infty(\dom)}\leq c||\bold U||_{W^{1,p}(\dom)},\quad \forall\;\bold U=\begin{pmatrix} \vec E\\ \vec H\end{pmatrix}\in W^{1,p}(\dom).
 \end{equation}
 
\begin{lem}\label{lem-lip}
Suppose $\bold U,\bold U'\in W^{1,p}(\dom)$, and $||\bold U||_{W^{1,p}(\dom)},||\bold U||_{W^{1,p}(\dom)}\leq \frac{s_0}{c}$, then
\begin{multline}
\left\Vert\F(\bold U)-\F(\bold U')\right\Vert_{W^{1,p}(\dom)}\\
\leq C\left(||\bold U||_{W^{1,p}(\dom)}^2+||\bold U'||_{W^{1,p}(\dom)}^2\right)||\bold U-\bold U'||_{W^{1,p}(\dom)},
\end{multline}
where $C>0$ does not depend on $\bold U$ and $\bold U'$. 
\end{lem} 
\begin{proof}
Suppose $\bold U= \begin{pmatrix} \vec E&\vec H\end{pmatrix}^t$, $\bold U'= \begin{pmatrix} \vec E'&\vec H'\end{pmatrix}^t$. Consider the difference
\begin{equation}\label{lem-a}
\X(|\vec E|^2)\vec E-\X(|\vec E'|^2)\vec E'=
\X(|\vec E|^2)(\vec E-\vec E')+\left(\X(|\vec E|^2)-\X(|\vec E'|^2)\right)\vec E'.
\end{equation}
Note that
\begin{equation}\label{split}
|\X(x,|\vec E(x)|^2)|\leq \sum_{k=1}^\infty ||a_k||_{L^\infty(\dom)}||\vec E||_{L^\infty(\dom)}^{2k}\leq 
M||\vec E||_{W^{1,p}(\dom)}^2,
\end{equation}
\begin{equation}
|(D_x\X)(x,|\vec E(x)|^2)|\leq M||\vec E||_{W^{1,p}(\dom)}^2,
\end{equation}
\begin{equation}
|(\pr_s\X)(x,|\vec E(x)|^2)|\leq \sum_{k=1}^\infty k||a_k||_{L^\infty(\dom)}||\vec E||_{L^\infty}^{2(k-1)}\leq M.
\end{equation}
Therefore
\begin{equation}\label{lem-b}
\left\Vert\X(|\vec E|^2)(\vec E-\vec E')\right\Vert_{L^p(\dom)}\leq C ||\vec E||_{W^{1,p}(\dom)}^2||\vec E-\vec E'||_{L^p(\dom)}.
\end{equation}
Also, since
\begin{multline}
D_x[\X(|\vec E|^2)(\vec E-\vec E')] =
\X(|\vec E|^2)D_x(\vec E-\vec E')\\
+(D_x\X)(|\vec E|^2)(\vec E-\vec E') +
2\re(\vec E^\ast\cdot D_x\vec E)(\pr_s\X)(|\vec E|^2)(\vec E-\vec E')
\end{multline}
and
\begin{equation}
\left\Vert \X(|\vec E|^2)D_x(\vec E-\vec E')\right\Vert_{L^p(\dom)} \leq M ||\vec E||_{W^{1,p}(\dom)}^2||D_x(\vec E-\vec E')||_{L^p(\dom)},
\end{equation}
\begin{equation}
\left\Vert(D_x\X)(|\vec E|^2)(\vec E-\vec E')\right\Vert_{L^p(\dom)} \leq M ||\vec E||_{W^{1,p}(\dom)}^2||\vec E-\vec E'||_{L^p(\dom)},
\end{equation}
\begin{multline}
\left\Vert2\re(\vec E^\ast\cdot D_x\vec E)(\pr_s\X)(|\vec E|^2)(\vec E-\vec E')\right\Vert_{L^p(\dom)}\\
\leq 2M ||\vec E||_{L^\infty(\dom)}||D_x\vec E||_{L^p(\dom)}||\vec E-\vec E'||_{L^\infty(\dom)}\\
\leq 2c^2M||\vec E||_{W^{1,p}(\dom)}^2||\vec E-\vec E'||_{W^{1,p}(\dom)},
\end{multline}
it follows that
\begin{equation}
\left\Vert \X(|\vec E|^2)(\vec E-\vec E')\right\Vert_{W^{1,p}(\dom)}\leq C ||\vec E||^2_{W^{1,p}(\dom)}||\vec E-\vec E'||_{W^{1,p}(\dom)}.
\end{equation}

In order to estimate the second term in \eqref{split}, let $\vec E_t=\vec E'+t(\vec E-\vec E')$. Then
\begin{equation}\label{int-diff}
\X(|\vec E|^2)-\X(|\vec E'|^2)=\int_0^1 \pr_s\X(|\vec E_t|^2)2\re(\vec E_t^\ast\cdot(\vec E-\vec E'))\dd t.
\end{equation}
We have, for $q=p$ or $q=\infty$, that
\begin{multline}
\left\Vert\pr_s\X(|\vec E_t|^2)2\re(\vec E_t^\ast\cdot(\vec E-\vec E'))\right\Vert_{L^q(\dom)}
\leq C||\vec E_t||_{W^{1,p}(\dom)}||\vec E-\vec E'||_{W^{1,p}(\dom)}\\
\leq C(||\vec E||_{W^{1,p}(\dom)}+||\vec E'||_{W^{1,p}(\dom)})||\vec E-\vec E'||_{W^{1,p}(\dom)}
\end{multline}
Now
\begin{multline}
D_x \left[\pr_s\X(|\vec E_t|^2)2\re(\vec E_t^\ast\cdot(\vec E-\vec E'))\right]
=(D_x\pr_s\X)(|\vec E_t|^2)2\re(\vec E_t^\ast\cdot(\vec E-\vec E'))\\
+\pr_s\X(|\vec E_t|^2)2\re(D_x\vec E_t^\ast\cdot(\vec E-\vec E'))\\
+\pr_s\X(|\vec E_t|^2)2\re(\vec E_t^\ast\cdot D_x(\vec E-\vec E'))\\
+\pr_s^2\X(|\vec E_t|^2)4\re(\vec E_t^\ast\cdot(\vec E-\vec E'))\re(\vec E_t^\ast\cdot D_x\vec E_t).
\end{multline}
Using the same type of estimates as above, we can obtain that
\begin{equation}
\left\Vert(D_x\pr_s\X)(|\vec E_t|^2)2\re(\vec E_t^\ast\cdot(\vec E-\vec E'))\right\Vert_{L^p(\dom)}
\leq C||\vec E_t||_{L^\infty(\dom)}||\vec E-\vec E'||_{L^p(\dom)},
\end{equation}
\begin{equation}
\left\Vert\pr_s\X(|\vec E_t|^2)2\re(D_x\vec E_t^\ast\cdot(\vec E-\vec E'))\right\Vert_{L^p(\dom)}
\leq C||\vec E_t||_{W^{1,p}(\dom)}||\vec E-\vec E'||_{L^\infty(\dom)},
\end{equation}
\begin{equation}
\left\Vert\pr_s\X(|\vec E_t|^2)2\re(\vec E_t^\ast\cdot D_x(\vec E-\vec E'))\right\Vert_{L^p(\dom)}
\leq C||\vec E_t||_{L^\infty(\dom)}||\vec E-\vec E'||_{W^{1,p}(\dom)},
\end{equation}
\begin{multline}
\left\Vert\pr_s^2\X(|\vec E_t|^2)4\re(\vec E_t^\ast\cdot(\vec E-\vec E'))\re(\vec E_t^\ast\cdot D_x\vec E_t)\right\Vert_{L^p(\dom)}\\
\leq C||\vec E_t||^2_{L^\infty(\dom)}||\vec E_t||_{W^{1,p}(\dom)}||\vec E-\vec E'||_{L^\infty(\dom)}\\
\leq C||\vec E_t||_{W^{1,p}(\dom)}||\vec E-\vec E'||_{L^\infty(\dom)}.
\end{multline}
Putting these together with \eqref{int-diff} it follows that
\begin{multline}\label{lem-c}
\left\Vert\left(\X(|\vec E|^2)-\X(|\vec E'|^2)\right)\vec E'\right\Vert_{W^{1,p}(\dom)}\\
\leq C\left( ||\vec E||_{W^{1,p}(\dom)}^2+||\vec E'||_{W^{1,p}(\dom)}^2\right)||\vec E-\vec E'||_{W^{1,p}(\dom)}.
\end{multline}
Equations \eqref{lem-a}, \eqref{lem-b}, \eqref{lem-c} imply
\begin{multline}
\left\Vert\X(|\vec E|^2)\vec E-\X(|\vec E'|^2)\vec E'\right\Vert_{W^{1,p}(\dom)}\\
\leq C\left( ||\vec E||_{W^{1,p}(\dom)}^2+||\vec E'||_{W^{1,p}(\dom)}^2\right)||\vec E-\vec E'||_{W^{1,p}(\dom)}.
\end{multline}
A similar estimate holds for the $\Y$ component of $\F$.
\end{proof}

Since $\F(0)=0$, Lemma \ref{lem-lip} has the corollary
\begin{cor}
If $\bold U\in W^{1,p}(\dom)$, and $||\bold U||_{W^{1,p}(\dom)}\leq \frac{s_0}{c}$, then
\begin{equation}
\left\Vert\F(\bold U)\right\Vert_{W^{1,p}(\dom)} \leq C||\bold U||_{W^{1,p}(\dom)}^3.
\end{equation}
\end{cor}

\begin{proof}[Proof of Theorem \ref{n-existence}] Applying Lemma \ref{lem-lip} and it's corollary together with Theorems \ref{h-existence} and \ref{inh-existence} we can show that the operator $ \T_{\vec f, \epsilon,\mu}$ defined in \eqref{T-def} is a contraction on a sufficiently small ball in $W^{1,p}_D(\dom)\times W^{1,p}_b(\dom)$, of radius $\m$, and therefore has a fixed point. 
\end{proof}

In the following discussion we will assume that $\m$ is chosen so that if $||\vec f||_{TW^{1-1/p,p}_{div}(\pr\dom)}<\m$, then 
\begin{equation}
||\bold U_0||_{W^{1,p}(\dom)}<\frac{\m}{2},
\end{equation}
and if $||\bold U||_{W^{1,p}(\dom)}<\m$, then
\begin{equation}
||\G_{\epsilon,\mu}(\F(\bold U))||_{W^{1,p}(\dom)}<\frac{\m}{2}.
\end{equation}

\subsection{Asymptotics}

For $\bold U\in W^{1,p}(\dom)=\begin{pmatrix}\vec E\\ \vec H\end{pmatrix}$, define
\begin{equation}
\F_k(\bold U)=\begin{pmatrix} b_k|\vec H|^{2k}\vec H\\ -a_k|\vec E|^{2k}\vec E\end{pmatrix},
\end{equation}
so
$
\F(\bold U)=\sum_{k=1}^\infty \F_k(\bold U).
$

Let $t$ be a small parameter. For  $\vec f\in TW^{1-1/p,p}_{div}(\pr\dom)$, let $\vec f^t=t\vec f$. Also, let $\bold U^t=\begin{pmatrix} \vec E^t \\ \vec H^t\end{pmatrix}$ be the solution of $\el_{\epsilon,\mu}\bold U^t=\F(\bold U^t)$ with boundary data $\ttr(\vec E_t)=\vec f$, and let $\bold U_0^t=\begin{pmatrix} \vec E^t_0 \\ \vec H^t_0\end{pmatrix}$ be the solution of $\el_{\epsilon,\mu}\bold U^t_0=0$ with the same boundary data.  Set $\bold U^t_k=\T^k_{\vec f,\epsilon,\mu}(\bold U_0^t)$, $k=1,2,\ldots$. For  $|t|<\m/||\vec f||_{TW^{1-1/p,p}_{div}(\pr\dom)}$, since $ \T_{\vec f, \epsilon,\mu}$ is a contraction, 
\begin{equation}
||U^t_k||_{W^{1,p}(\dom)}\leq\m\text{ and }\bold U^t_k\to \bold U^t\text{ in }W^{1,p}(\dom),\text{ as }k\to\infty.
\end{equation}

Observe that $\bold U^t_0=t\bold U_0$. Define
\begin{equation}
\bold V^t_1=\begin{pmatrix} \vec B_1^t\\ \vec A^t_1\end{pmatrix}=\bold U^t_1-\bold U_0^t=\G_{\epsilon,\mu}(\F(\bold U_0^t)),
\end{equation}
\begin{equation}
\bold V^t_k=\begin{pmatrix} \vec B_k^t\\ \vec A^t_k\end{pmatrix}=\bold U^t_k-\bold U^t_{k-1}=\G_{\epsilon,\mu}\left(\F(\bold U_{k-2}^t+\bold V^t_{k-1}) -\F(\bold U^t_{k-2})\right).
\end{equation}
Then
\begin{equation}
\bold V^t_1=t^3\G_{\epsilon,\mu}(\F_1(\bold U_0))+t^5\G_{\epsilon,\mu}(\F_2(\bold U_0))+\cdots,
\end{equation}
\begin{multline}
\bold V^t_2= \G_{\epsilon,\mu}(\F_1(t\bold U_0+\bold V^t_1)-\F_1(t\bold U_0))+\cdots\\
=\G_{\epsilon,\mu} ( t^2\begin{pmatrix} b_1|\vec H_0|^2\vec B^t_1 +2\vec H_0\re(\vec H_0\cdot\vec B_1^{t\ast})\\ - a_1|\vec E_0|^2\vec A^t_1 -2\vec E_0\re(\vec E_0\cdot\vec A_1^{t\ast})\end{pmatrix} )+\cdots = \Ord(t^5)
\end{multline}
and so on.

\begin{lem}\label{Vk-estimate}
\begin{equation}
||\bold V^t_k||_{W^{1,p}(\dom)}=\Ord(t^{2k+1}),\text{ as }t\to0.
\end{equation}
\end{lem}
\begin{proof}
Follows easily by induction.
\end{proof}

\begin{lem}
\begin{equation}
||\bold U^t-\bold U_k^t||_{W^{1,p}(\dom)}=\Ord(t^{2k+3}),\text{ as }t\to0.
\end{equation}
\end{lem}
\begin{proof}
Let $\bold u^t_k=\bold U^t-\bold U_k^t$. Then
\begin{multline}
\bold u^t_k = \G_{\epsilon,\mu}(\F(\bold U^t_k+\bold u_k^t)-\F(\bold U^t_{k-1}))\\
=\G_{\epsilon,\mu}(\F(\bold U^t_k+\bold u_k^t)-\F(\bold U^t_{k})) 
+\G_{\epsilon,\mu}(\F(\bold U^t_k)-\F(\bold U^t_{k-1}))\\
=\left(\T_{\vec f,\epsilon,\mu}(\bold U^t_k+\bold u^t_k)-\T_{\vec f,\epsilon,\mu}(\bold U^t_k)\right)+\bold V^t_{k+1}.
\end{multline}
Since for small enough $t$, $\T_{\vec f,\epsilon,\mu}$ is a contraction, the first term on the right hand side may be absorbed into the left hand side and applying Lemma \ref{Vk-estimate}, the result follows.
\end{proof}

Notice that the terms multiplying $t^{2k+1}$ are the same for all $\bold U_{k'}^t$ with $k'\geq k$. Define then
\begin{equation}
\bold W_k=\left.\frac{1}{(2k+1)!}\pr_t^{2k+1}\bold U_k^t\right|_{t=0}.
\end{equation}
A useful observation is that
\begin{equation}
\bold W_k=\G_{\epsilon,\mu}(\begin{pmatrix} b_k|\vec H_0|^{2k}H_0\\ -a_k|\vec E_0|^{2k}E_0\end{pmatrix})+\left(\text{terms constructed from $\{a_l,b_l\}_{l=1}^{k-1}$ and $\bold U_0$}  \right),
\end{equation}
so
\begin{equation}\label{wk-nonlinearity}
\el_{\epsilon,\mu}\bold W_k=\begin{pmatrix} b_k|\vec H_0|^{2k}H_0\\ -a_k|\vec E_0|^{2k}E_0\end{pmatrix}+\left(\text{terms constructed from $\{a_l,b_l\}_{l=1}^{k-1}$ and $\bold U_0$}  \right).
\end{equation}

\section{The inverse problem}\label{inverse}

Suppose $\mathcal{B}_{\epsilon,\mu,\F}=\mathcal{B}_{\epsilon',\mu',\F'}$. For any $\vec f\in TW^{1-1/p,p}_{div}(\pr\dom)$ let $\bold U_0=\begin{pmatrix} \vec E_0\\ \vec H_0\end{pmatrix}$, $\bold U_0'=\begin{pmatrix} \vec E_0'\\ \vec H_0'\end{pmatrix}$, $\bold W_k= \begin{pmatrix} \vec E_k\\ \vec H_k\end{pmatrix}$, $\bold W_k'= \begin{pmatrix} \vec E_k'\\ \vec H_k'\end{pmatrix}$ be the constructed as in the previous sections using the two sets of coefficients respectively. Then we have
\begin{equation}
\ttr(\vec E_0)=\ttr(\vec E_0')=\vec f,\quad \ttr(\vec H_0)=\ttr(\vec H_0'),
\end{equation}
\begin{equation}\label{same-bv}
\ttr(\vec E_k)=\ttr(\vec E_k')=0,\quad \ttr(\vec H_k)=\ttr(\vec H_k'),\quad k=1,2,\ldots.
\end{equation}
An immediate consequence is that
\begin{multline}
\{ (\ttr(\vec E_0),\ttr(\vec H_0))\in TW^{1-1/p,p}_{div}(\pr\dom)\times TW^{1-1/p,p}_{div}(\pr\dom):\el_{\epsilon,\mu}{\begin{pmatrix}\vec E_0\\ \vec H_0\end{pmatrix}}=0\}\\
=\{ (\ttr(\vec E_0'),\ttr(\vec H_0'))\in TW^{1-1/p,p}_{div}(\pr\dom)\times TW^{1-1/p,p}_{div}(\pr\dom):\el_{\epsilon',\mu'}{\begin{pmatrix}\vec E_0'\\ \vec H_0'\end{pmatrix}}=0\}.
\end{multline}
It is a known result then (e.g. see \cite{OS}, or \cite{C}) that $\epsilon=\epsilon'$ and $\mu=\mu'$. It follows that $\bold U_0=\bold U_0'$.

{Suppose then that $\{a_l,b_l\}_{l=1}^{k-1}=\{a_l',b_l'\}_{l=1}^{k-1}$. We will show that then $a_k=a_k'$ and $b_k=b_k'$. Theorem \ref{uniqueness} will follow by induction.}

\subsection{An integral identity}

For two vector fields $\vec A$ and $\vec B$, we have
\begin{equation}
\int_{\pr\dom}(\vec n\times \vec A)\cdot\vec B= \int_\dom (\nabla\times\vec A)\cdot\vec B-
\vec A\cdot(\nabla\times\vec B).
\end{equation}
Let $\bold u_0=\begin{pmatrix}\vec e_0\\\vec h_0\end{pmatrix}$ be a solution of $\el_{\epsilon,\mu}\bold u_0=0$. Using \eqref{wk-nonlinearity} we get
\begin{multline}
\int_{\pr\dom}(\vec n\times \vec E_k)\cdot \vec h_0= \int_\dom i\omega\mu\vec H_k\cdot\vec h_0+b_k|\vec H_0|^{2k}\vec H_0\cdot\vec h_0 +i\omega\epsilon\vec E_k\cdot\vec e_0\\
+\int_\dom\left(\text{terms constructed from $\{a_l,b_l\}_{l=1}^{k-1}$, $\bold U_0$, and $\bold u_0$}  \right),
\end{multline}
\begin{multline}
\int_{\pr\dom}(\vec n\times \vec H_k)\cdot \vec e_0= -\int_\dom i\omega\epsilon\vec E_k\cdot\vec e_0+a_k|\vec E_0|^{2k}\vec E_0\cdot\vec e_0 +i\omega\mu\vec H_k\cdot\vec h_0\\
+\int_\dom\left(\text{terms constructed from $\{a_l,b_l\}_{l=1}^{k-1}$, $\bold U_0$, and $\bold u_0$}  \right).
\end{multline}
Then
\begin{multline}
\int_{\pr\dom}(\vec n\times \vec E_k)\cdot \vec h_0+(\vec n\times \vec H_k)\cdot \vec e_0=
\int_\dom b_k|\vec H_0|^{2k}\vec H_0\cdot\vec h_0 - a_k|\vec E_0|^{2k}\vec E_0\cdot\vec e_0\\
+\int_\dom\left(\text{terms constructed from $\{a_l,b_l\}_{l=1}^{k-1}$, $\bold U_0$, and $\bold u_0$}  \right).
\end{multline}
Subtracting the corresponding identities for the components of $\bold W_k'$ and using \eqref{same-bv} we have
\begin{equation}
I_k(\bold U_0,u_0)= \int_\dom (b_k-b_k')|\vec H_0|^{2k}\vec H_0\cdot\vec h_0 - (a_k-a_k')|\vec E_0|^{2k}\vec E_0\cdot\vec e_0=0,
\end{equation}
which holds for any $\bold u_0$, $\bold U_0$ solutions of the homogeneous linear equation.

Let $\bold u_j =\begin{pmatrix} \vec e_j\\\vec h_j\end{pmatrix}\in W^{1,p}_{b}(\dom)\times W^{1,p}_{b}(\dom)$ all satisfy $\el_{\epsilon,\mu}\bold u_j=0$, $j=1,2,3$, then
\begin{equation}
I_k(t_1\bold u_1+t_2\bold u_2+t_3\bold u_3,u_0)=0,\quad \forall t_1,t_2,t_3\in\C.
\end{equation}
The left hand side of this identity is  a polynomial in $t_1,t_2,t_3,t_1^\ast,t_2^\ast,t_3^\ast$, so the coefficient of  each independent monomial must vanish. In particular, the coefficient of $t_1 t_2^kt_3^{\ast k}$ must be zero.
The vanishing quantity is
\begin{multline}\label{integral-identity}
\int_\dom (b_k-b'_k)\left[\vec h_0\cdot\vec h_1 (\vec h_2\cdot\vec h_3^\ast)^k+k\vec h_0\cdot\vec h_2(\vec h_1\cdot\vec h_3^\ast) (\vec h_2\cdot\vec h_3^\ast)^{k-1}\right]\\
-\int_\dom (a_k-a'_k)\left[\vec e_0\cdot\vec e_1 (\vec e_2\cdot\vec e_3^\ast)^k+k\vec e_0\cdot\vec e_2(\vec e_1\cdot\vec e_3^\ast) (\vec e_2\cdot\vec e_3^\ast)^{k-1}\right]=0
\end{multline}

\subsection{CGO solutions for the linear Maxwell system}

CGO solutions for the linear Maxwell system have been constructed in many past works. The method given here is due to \cite{OPS}, \cite{OS}. We will mostly follow the construction as given in \cite{C}, summarizing the results when the argument proceeds identically and giving more detail when not. We show that
\begin{prop}\label{cgo}
There exists a constant $C(\rho, ||\epsilon||_{W^{5,\infty}(\dom)}, ||\mu||_{W^{5,\infty}(\dom)})>0$ such that if ${\vec\zeta}\in \C^3$, ${\vec\zeta}\cdot{\vec\zeta}=\omega^2$,
\begin{equation}
|{\vec\zeta}|>C(\rho, ||\epsilon||_{W^{5,\infty}(\dom)}, ||\mu||_{W^{5,\infty}(\dom)}),
\end{equation}
then there exist solutions $\bold U=\begin{pmatrix} \vec E\\\vec H\end{pmatrix}$ of $\el_{\epsilon,\mu}\bold U=0$ such that
\begin{equation}
\vec E= e^{i\vec \zeta\cdot x}\left(\sigma_e\epsilon^{-1/2}\frac{\vec \zeta}{|\zeta|}+\vec r_e\right),
\end{equation}
\begin{equation}
\vec H= e^{i\vec \zeta\cdot x}\left(\sigma_h\mu^{-1/2}\frac{\vec \zeta}{|\zeta|}+\vec r_h\right),
\end{equation}
\begin{equation}
||\vec r_e||_{L^\infty(\dom)},||\vec r_h||_{L^\infty(\dom)} =\Ord(|\vec \zeta|^{-1}),
\end{equation}
and $\sigma_e,\sigma_h\in\{0,1\}$.
\end{prop}

Let $\alpha=\log \epsilon$, $\beta = \log \mu$, and $I_n$ be the identity matrix in dimension $n$. 
Suppose that
\begin{equation}
X=\left(\begin{array}{c}
h\\ \vec H \\\hline e\\\vec E
\end{array}\right)
\end{equation}
satisfies the equation
\begin{equation}
(P+V)X=0,
\end{equation}
where
\begin{equation}
P=\frac{1}{i}\left(\begin{array}{cc|cc}
 & & &\dv\\ & & \nabla&-\crl\\\hline
  &\dv& & \\ \nabla& \crl& & 
\end{array}\right),
\end{equation}
\begin{equation}
V=\frac{1}{i}\left(\begin{array}{cc|cc}
i\omega\mu& & &(\nabla\alpha)\cdot\\ & i\omega\mu I_3& (\nabla \alpha)& \\\hline
 & (\nabla \beta)\cdot&i\omega\epsilon \\ (\nabla\beta)& & &i\omega\epsilon I_3
\end{array}\right).
\end{equation}
Observe that if $e$ and $h$ vanish, then$(\vec E, \vec H)$ is a solution of 
\begin{equation}
\left\{\begin{array}{l}
\crl\vec E= i\omega\mu \vec H,\\[5pt]
\crl\vec H=- i\omega\epsilon \vec E.
\end{array}\right.
\end{equation}
Let
\begin{equation}
Y=\left(\begin{array}{c|c}
\mu^{1/2}I_4& \\\hline
 &\epsilon^{1/2} I_4
\end{array}\right) X,\quad\kappa =\omega\mu^{1/2}\epsilon^{1/2}, 
\end{equation}
\begin{equation}
W=\kappa I_8+
\frac{1}{2i}\left(\begin{array}{cc|cc}
 & & &(\nabla\alpha)\cdot\\ & & (\nabla \alpha)& (\nabla\alpha)\times \\\hline
 & (\nabla \beta)\cdot& & \\ (\nabla\beta)& -(\nabla\beta)\times & &
\end{array}\right).
\end{equation}
Then 
\begin{equation}\label{Y-eq}
(P+W)Y=0.
\end{equation}

Note that
\begin{equation}
(P+W)(P-W^t)=-\triangle +Q, 
\end{equation}
where
\begin{multline}
Q=\frac{1}{2}\left(\begin{array}{cc|cc}
(\triangle\alpha)& & & \\ & 2(\pr_i\pr_j\alpha)_{ij} - (\triangle\alpha)I_3& & \\\hline
 & &(\triangle\beta) & \\ & & &  2(\pr_i\pr_j\beta)_{ij} - (\triangle\beta)I_3
\end{array}\right)\\[5pt]
-\left(\begin{array}{c|c}
(\kappa^2-\frac{1}{4}(\nabla\alpha\cdot\nabla\alpha))I_4& \begin{matrix} & -2i(\nabla\kappa)\cdot\\ -2i(\nabla\kappa)& \end{matrix} \\\hline
\begin{matrix} & -2i(\nabla\kappa)\cdot\\ -2i(\nabla\kappa)& \end{matrix}&
(\kappa^2-\frac{1}{4}(\nabla\beta\cdot\nabla\beta))I_4
\end{array}\right).
\end{multline}
If $Z$ is a solution of
\begin{equation}\label{Z-eq}
(-\triangle +Q)Z=0,
\end{equation}
then $Y=(P-W^t)Z$ is a solution to \eqref{Y-eq}. We would like to construct solutions of \eqref{Z-eq} that are of the form 
\begin{equation}
Z({\vec\zeta},x)=e^{i{\vec\zeta}\cdot x}(L({\vec\zeta})+R({\vec\zeta},x)),\quad {\vec\zeta}\in\C^3.
\end{equation}
To do so, first extend the coefficients $\epsilon$, $\mu$ to $\R^3$ so that $\epsilon-1,\mu-1\in C^5_0(\R^3)$. Then $\omega^2I_8+Q\in C^3_0(\R^3)$. Let $\rho>0$ be such that $\supp(\omega^2I_8+Q)$ is contained in the ball of radius $\rho$. We can prove the following

\begin{lem}[compare to {\cite[Lemma 8]{C}}]
There exist a $C(\rho)>0$ such that for any $L\in \C^8$,   ${\vec\zeta}\in\C^3$ with ${\vec\zeta}\cdot{\vec\zeta}=\omega^2$ and
\begin{equation}\label{zeta-condition}
|{\vec\zeta}|>C(\rho)||\omega^2 I_8+Q||_{L^\infty(\R^3)},
\end{equation}
there exists 
\begin{equation}
Z=e^{i{\vec\zeta}\cdot x}(L+R)
\end{equation}
a solution of \eqref{Z-eq} in $\R^3$, $Z\in W^{3,2}(\dom)$ and with
\begin{equation}\label{R-bound}
||R||_{W^{3,2}(\dom)}\leq \frac{1}{|{\vec\zeta}|}C(\rho)|L|\,||\omega^2+Q||_{W^{3,\infty}(\R^3)}.
\end{equation}
\end{lem}
\begin{proof}
We only need to show that such an $R$ exists. The equation it need to satisfy is
\begin{equation}
(-\triangle-2i{\vec\zeta}\cdot\nabla)R+(\omega^2I_8+Q)R=-(\omega^2I_8+Q)L.
\end{equation}

We would like to, in a certain sense, invert $(-\triangle-2i{\vec\zeta}\cdot\nabla)$. For some $-1<\delta<0$, define the spaces
\begin{equation}
L^2_\delta(\R^3)=\left\{f:||f||_{L^2_\delta}=||(1+|x|^2)^{\delta/2}f||_{L^2(\R^3)}<\infty\right\},
 \end{equation} 
 \begin{equation}
W^{s,2}_\delta(\R^3)=\left\{f:||f||_{W^{s,2}_\delta}=||(1+|x|^2)^{\delta/2}f||_{W^{s,2}(\R^3)}<\infty\right\}.
 \end{equation} 
There exists (see, for example, \cite[Corollary 2.2]{SU}) $G_{\vec\zeta}:W^{s,2}_{\delta+1}(\R^3)\to W^{s,2}_\delta(\R^3)$ such that $(-\triangle-2i{\vec\zeta}\cdot\nabla) G_{\vec\zeta} \phi=\phi$ and
\begin{equation}
|| G_{\vec\zeta} \phi||_{W^{s,2}_\delta(\R^3)}\leq \frac{1}{|{\vec\zeta}|}C(\delta)||f||_{W^{s,2}_{\delta+1}(\R^3)}
\end{equation}

The equation $R$ should satisfy can then be written as
\begin{equation}
(I_8+G_{\vec\zeta}(\omega^2I_8+Q))R=-G_{\vec\zeta}(\omega^2 I_8+Q)L.
\end{equation}
We can choose the constant $C(\rho)$ in \eqref{zeta-condition} so that
\begin{equation}
||G_{\vec\zeta}(\omega^2I_8+Q)R||_{W^{3,2}_\delta(\R^3)}\leq \frac{1}{2}||R||_{W^{3,2}_\delta(\R^3)},
\end{equation}
in which case there exists a solution
\begin{equation}
R=-\left(I_8+G_{\vec\zeta}(\omega^2I_8+Q)\right)^{-1}G_{\vec\zeta}(\omega^2 I_8+Q)L,
\end{equation}
and it satisfies the estimate \eqref{R-bound}.
\end{proof}

The following lemma is a restatement of a result in \cite{C}:
\begin{lem}[see {\cite[Proposition 9]{C}}]
There exists a constant $C(\rho, ||\epsilon-1||_{W^{5,\infty}(\R^3)}, ||\mu-1||_{W^{5,\infty}(\R^3)})>0$ such that if ${\vec\zeta}\in \C^3$, ${\vec\zeta}\cdot{\vec\zeta}=\omega^2$,
\begin{equation}
|{\vec\zeta}|>C(\rho, ||\epsilon-1||_{W^{5,\infty}(\R^3)}, ||\mu-1||_{W^{5,\infty}(\R^3)}),
\end{equation}
\begin{equation}
L=\frac{1}{|{\vec\zeta}|}\left(\begin{array}{c}
{\vec\zeta}\cdot \vec a\\ \omega \vec b\\\hline {\vec\zeta}\cdot \vec b\\ \omega \vec a
\end{array}\right),\quad \vec a,\vec b\in C^3,
\end{equation}
then there exists 
\begin{equation}
Z=e^{i{\vec\zeta}\cdot x}(L+R)
\end{equation}
a solution of \eqref{Z-eq} in $\R^3$, $Z\in W^{3,2}(\dom)$ and with
\begin{equation}\label{R-bound}
||R||_{W^{3,2}(\dom)}\leq \frac{1}{|{\vec\zeta}|}C(\rho)|L|\,||\omega^2+Q||_{W^{3,\infty}(\R^3)}.
\end{equation}
Additionally, $Y=(P-W^t)Z$ solves $(P+W)Y=0$ and is of the form
\begin{equation}
Y=\left(\begin{array}{c}
0\\ \mu^{1/2}\vec H \\\hline 0\\\epsilon^{1/2}\vec E
\end{array}\right).
\end{equation}
\end{lem}

Under the conditions of the previous lemma, we get 
\begin{equation}
\vec E=e^{i{\vec\zeta}\cdot x}\left(\epsilon^{-1/2}\frac{{\vec\zeta}\cdot\vec a}{|{\vec\zeta}|}{\vec\zeta}+\vec r_e\right),
\end{equation}
\begin{equation}
\vec H=e^{i{\vec\zeta}\cdot x}\left(\mu^{-1/2}\frac{{\vec\zeta}\cdot\vec b}{|{\vec\zeta}|}{\vec\zeta}+\vec r_h\right).
\end{equation}
For $\sigma_e,\sigma_h\in\{0,1\}$, choose
\begin{equation}
\vec a = \sigma_e \frac{\vec\zeta^\ast}{|\vec\zeta|^2},\quad 
\vec b = \sigma_h \frac{\vec\zeta^\ast}{|\vec\zeta|^2}.
\end{equation}
Then
\begin{equation}
\vec E= e^{i\vec \zeta\cdot x}\left(\sigma_e\epsilon^{-1/2}\frac{\vec \zeta}{|\zeta|}+\vec r_e\right),
\end{equation}
\begin{equation}
\vec H= e^{i\vec \zeta\cdot x}\left(\sigma_h\mu^{-1/2}\frac{\vec \zeta}{|\zeta|}+\vec r_h\right),
\end{equation}
and, applying the previous lemma and Sobolev embedding
\begin{equation}
||\vec r_e||_{L^\infty(\dom)},||\vec r_h||_{L^\infty(\dom)} =\Ord(|\vec \zeta|^{-1}).
\end{equation}

\subsection{Proof of the main theorem}

Let $\xi\in\R^3$. WLOG,  $\xi=\xi_1e_1$. Let $\vec \zeta_j=\alpha_j+i\beta_j\in\C^n$, $j=0,1,2,3$,
\begin{equation}
\beta_j=(-1)^j\left(\tau^2+\xi_1^2/4  \right)^{1/2}e_3, 
\end{equation}
\begin{equation}
\alpha_0=\frac{\xi_1}{2}e_1-(\omega^2+\tau^2)^{1/2}e_2,
\end{equation}
\begin{equation}
\alpha_1=\frac{\xi_1}{2}e_1+(\omega^2+\tau^2)^{1/2}e_2,
\end{equation}
\begin{equation}
\alpha_2=\alpha_3=-\frac{\xi_1}{2}e_1-(\omega^2+\tau^2)^{1/2}e_2.
\end{equation}
Then
\begin{equation}
\vec \zeta_j\cdot\vec \zeta_j=\omega^2,\quad |\vec \zeta_j|^2=2\tau^2+\omega^2+\xi_1^2/4.
\end{equation}
For sufficiently large $\tau>0$, let
\begin{equation}
\vec e_j= e^{i\vec \zeta_j\cdot x}(\sigma_e\epsilon^{-1/2}|\vec\zeta_j|^{-1}\vec \zeta_j+r_{ej})
\end{equation}
\begin{equation}
\vec h_j= e^{i\vec \zeta_j\cdot x}(\sigma_h\mu^{-1/2}|\vec\zeta_j|^{-1}\vec \zeta_j+r_{hj})
\end{equation}
be the special solutions given by Proposition \ref{cgo}.

Note that
\begin{equation}
\vec \zeta_0\cdot\vec \zeta_1=-\omega^2-2\tau^2,\quad \vec\zeta_2\cdot\vec \zeta_3^\ast=2\tau^2+\omega^2+\xi_1^2/4,
\end{equation}
\begin{equation}
\vec \zeta_0\cdot\vec\zeta_2=2\tau^2+\omega^2,\quad
\vec \zeta_1\cdot\vec \zeta_3=-\omega^2.
\end{equation}
Then
\begin{equation}
\vec e_0\cdot\vec e_1(\vec e_2\cdot \vec e_3^\ast)^k=-
\sigma_e|\epsilon|^{-k}\epsilon^{-1}\exp(i\xi\cdot x)+\Ord(\tau^{-1}),
\end{equation}
\begin{equation}
\vec e_0\cdot\vec e_2(\vec e_1\cdot\vec e_3)(\vec e_2\cdot \vec e_3^\ast)^{k-1}=\Ord(\tau^{-1}),
\end{equation}
where $\Ord(\tau^{-1})$ is to be understood in the sense of $L^\infty(\dom)$ norms. Choosing $\sigma_e=1$, $\sigma_h=0$ and taking the limit $\tau\to\infty$ in \eqref{integral-identity}, we get
\begin{equation}
\left(\frac{a_k-a'_k}{|\epsilon|^{k}\epsilon}\chi_\dom  \right)^{\wedge}(\xi)=0,\quad \forall \xi\in\R^3.
\end{equation}
This implies $a_k=a'_k$. By an identical argument it follows that $b_k=b'_k$. This concludes the induction step.

\paragraph{Acknowledgement} The author is grateful to Prof. Gunther Uhlmann for proposing this problem and for suggesting improvements to the manuscript.

\bibliography{semilinear}
\bibliographystyle{plain}
\end{document}